\documentclass[a4paper,12pt]{amsart}
\usepackage{amssymb,}
\usepackage[top=3cm,bottom=3cm,outer=3cm,inner=3cm,marginpar=3cm]{geometry}
\usepackage[utf8]{inputenc}
\usepackage{hyperref}

\newtheorem{theorem}{Theorem}[section]

\newtheorem{corollary}[theorem]{Corollary}
\newtheorem{proposition}[theorem]{Proposition}

\theoremstyle{definition}
\newtheorem{definition}{Definition}

\newtheorem{remark}[theorem]{Remark}

\numberwithin{equation}{section}

\def\RR{\mathbb{R}} 
\def\CC{\mathbb{C}} 
\def\NN{\mathbb{N}} 
\def\ddbar{\partial\bar\partial} %
\def\ddc{{i\partial\bar\partial}} %
 
\def\ric{\mathrm{Ric}}
\def\ke{K\"ahler-Einstein }
\def\ma{Monge-Amp\`{e}re }

\newcommand{\paren}[1]{\left(#1\right)}

\newcommand{\pd}[2]{\frac{\partial#1}{\partial#2}}
\newcommand{\abs}[1]{\left\vert#1\right\vert}
\newcommand{\norm}[1]{\left\|#1\right\|}

\newcommand{\ov}[1]{\overline{#1}}
\newcommand{\inner}[1]{\left\langle{#1}\right\rangle}

\begin{document}
\title[Fiberwise K\"{a}hler-Ricci flows]{Fiberwise K\"{a}hler-Ricci flows on families of bounded strongly pseudoconvex domains}

\author{Young-Jun Choi}
\address{Department of Mathematics, Pusan National University, 2, Busandaehak-ro 63beon-gil, Geumjeong-gu, Busan 46241, Republic of Korea}
\email{youngjun.choi@pusan.ac.kr}

\author{Sungmin Yoo}
\address{Center for Geometry and Physics, Institute for Basic Science (IBS), Pohang 37673, Republic of Korea}
\email{sungmin@ibs.re.kr}


\begin{abstract}
Let $\pi:\CC^n\times\CC\rightarrow \CC$ be the projection map onto the second factor and let $D$ be a domain in $\CC^{n+1}$ such that for $y\in\pi(D)$, every fiber $D_y:=D\cap\pi^{-1}(y)$ is a smoothly bounded strongly pseudoconvex domain in $\CC^n$ and is diffeomorphic to each other.
By Chau's theorem, the K\"ahler-Ricci flow has a long time solution $\omega_y(t)$ on each fiber $X_y$.
This family of flows induces a smooth real (1,1)-form $\omega(t)$ on the total space $D$ whose restriction to the fiber $D_y$ satisfies $\omega(t)\vert_{D_y}=\omega_y(t)$.
In this paper, we prove that $\omega(t)$ is positive for all $t>0$ in $D$ if $\omega(0)$ is positive.
As a corollary, we also prove that the fiberwise K\"ahler-Einstein metric is positive semi-definite on $D$ if $D$ is pseudoconvex in $\CC^{n+1}$.
\end{abstract}

\maketitle

\section{Introduction}

Let $D$ be a domain in $\CC^{n+1}$ and $S:=\pi(D)\subset\CC$, where $\pi:\CC^n\times\CC\rightarrow\CC$ is the standard projection map onto the second factor. 
We say that $D$ is a \emph{holomorphic family of bounded strongly pseudoconvex domains} if it satisfies the following:
\begin{itemize}

\item[(\romannumeral1)] $\pi^{-1}(S)\cap\partial D$ is smooth  and $\pi\vert_{\partial D}:\pi^{-1}(S)\cap\partial D\rightarrow S$ is a submersion. 

\item[(\romannumeral2)] For $y\in S$, all fibers $D_y:=\pi^{-1}(y)\cap D$ are smoothly bounded strongly pseudoconvex domains in $\CC^{n}$.

\end{itemize}
In this case, there exists a defining function $r$ of $D$ such that $\omega:=i\partial\ov\partial(-\log(-r))$ is a $d$-closed smooth real $(1,1)$-form on $D$ whose restriction to the fibers $\omega\vert_{D_y}$ is a complete K\"ahler metric with bounded geometry (see Section \ref{S:initial}).

Now we consider the following (normalized) K\"ahler-Ricci flow on each fiber $D_y$:
\begin{equation}\label{E:KR_flow}
\begin{aligned}
\frac{\partial}{\partial t}\omega_y(t)
&=
-\ric(\omega_y(t))-(n+1)\omega_y(t),\medskip\\
\omega_y(0)
&=
\omega|_{D_y}.
\end{aligned}
\end{equation}
This flow has a long time solution $\omega_y(t)$ which converges to the unique complete K\"ahler-Einstein metric $\omega^{KE}_y$ with Ricci curvature $-(n+1)$ as $t\rightarrow\infty$ by Chau's theorem in \cite{Chau}.
In fact, $\omega_y(t)$ is given by the solution of a parabolic \ma equation.
As a consequence of the implicit function theorem for the \ma operator, we obtain smooth real (1,1)-forms $\omega(t)$ on the total space $D$ whose restriction to the fibers $D_y$ satisfies $\omega(t)|_{D_y}=\omega_y(t)$ (see Proposition \ref{P:Regularity}).
Moreover, $\omega(t)$ evolves by the following equation, called the \emph{fiberwise K\"ahler-Ricci flow}:
\begin{equation}
\begin{aligned}
\frac{\partial}{\partial t}\omega(t)
&=
\Theta_{\omega(t)}-(n+1)\omega(t),\medskip\\
\omega(0)
&=
\omega,
\end{aligned}
\end{equation}
where $\Theta_{\omega(t)}$ is the relative curvature form of $\omega(t)$ (Theorem \ref{T:Relative Kahler-Ricci flow}).
This flow was first introduced by Berman in \cite{Berman} for the case of compact fibrations with the name ``\emph{relative K\"ahler-Ricci flow}".
The main theorem of this paper is a non-compact version of Berman's theorem (cf. Corollary 4.9 in \cite{Berman}).

\begin{theorem}\label{T:Main}
If $\omega$ is semi-positive in $D$ and strictly positive at least one point on each fiber $D_y$, then $\omega(t)$ is positive in $D$ for all $t>0$.
\end{theorem}

On the other hand, the family of K\"ahler-Einstein metrics $\omega^{KE}_y$ on $D_y$ also induces a $d$-closed smooth real $(1,1)$-form $\rho$ on the total space $D$.  
The form $\rho$ is called the \emph{fiberwise K\"ahler-Einstein metric} since it satisfies $\rho|_{D_y}=\omega^{KE}_y$ (cf. \cite{Choi1,Schumacher}).
Using the fact that $\omega_y(t)$ converges to $\omega^{KE}_y$ on each fiber $D_y$, one can show that the solution of the fiberwise K\"{a}hler-Ricci flow $\omega(t)$ smoothly converges to the fiberwise K\"{a}hler-Einstein metric $\rho$ on the total space $D$ (Theorem \ref{T:convergence of fiberwise}).
Since the existence of initial form $\omega$ satisfying the hypothesis in Theorem \ref{T:Main} is guaranteed provided that $D$ is pseudoconvex in $\CC^{n+1}$ (Proposition \ref{P:Pseudoconvex}),
we have the following

\begin{corollary}\label{C:Main}
The fiberwise K\"ahler-Einstein metric $\rho$ is semi-positive if $D$ is pseudoconvex.
\end{corollary}

Corollary \ref{C:Main} has already proved by the first named author in \cite{Choi1, Choi2}.
In fact, he further proved that $\rho$ is strictly positive if $D$ is strongly pseudoconvex.
In his papers \cite{Choi1, Choi2}, he analyzed the boundary behavior of the variation of K\"ahler-Eintein metrics via the boundary behavior of K\"ahler-Einstein metric due to Cheng and Yau.
It is remarkable to note that the analysis for Corollary \ref{C:Main} in this paper is lighter than the one in \cite{Choi1, Choi2}.
\medskip

A study on the positive variation of K\"ahler-Einstein metrics is first developed by Schumacher \cite{Schumacher}.
More precisely, he has proved that the variation of K\"ahler-Einstein metrics on a family of canonically polarized compact K\"ahler manifolds is positive-definite on the total space.
In \cite{Schumacher}, he showed that the geodesic curvature of the fiberwise K\"ahler-Einstein metric, which measures the positivity along the horizontal direction, satisfies a certain elliptic partial differential equation. 
A direct application of maximum principle says that the geodesic curvature is positive, which is equivalent to the positivity of the fiberwise K\"ahler-Einstein metric.

Later, Berman \cite{Berman} proved the parabolic version of Schumacher's result in the same setting.
On a canonically polarized compact K\"ahler manifold, the K\"ahler-Ricci flow has a long time solution which converges to the unique K\"ahler-Einstein metric by Cao's theorem in \cite{Cao}.
Using this result, Berman constructed the relative K\"ahler-Ricci flow on a family of canonically polarized compact K\"ahler manifolds. 
In \cite{Berman}, he proved the geodesic curvature of the relative K\"ahler-Ricci flow satisfies a parabolic version of Schumacher's elliptic PDE.
A parabolic maximum principle implies that the positivity of the relative K\"ahler-Ricci flow is preserved.
In particular, Berman's result implies the Schumacher's one since the relative K\"ahler-Ricci flow converges to the fiberwise K\"ahler-Einstein metric.
\medskip

In this paper, we shall generalize Berman's results to a family of bounded strongly pseudoconvex domains, which is one of the most important examples for non-compact complete K\"ahler manifolds.
In this case, the K\"ahler-Ricci flow has a long time solution which converges to the unique K\"{a}hler-Einstein metric due to Chau \cite{Chau}.
Moreover, the geodesic curvature of the fiberwise K\"{a}hler-Ricci flow still satisfies Berman's parabolic PDE.
The difference comes from applying the parabolic maximum principle.
In the previous case, since every fiber is compact, we can apply the standard weak and strong parabolic maximum principle.
However, if the manifold is non-compact, the weak maximum principle does not hold in general.

To resolve this problem, we will use Ni's theorem in \cite{Ni}, which says that if the function does not blow up too fast at the point at infinity, then the weak maximum principle holds.
To apply this, we have to investigate the boundary behavior of the geodesic curvature of the fiberwise K\"ahler-Ricci flow.
In fact, we will show that it has a polynomial growth near the boundary with respect to the defining function.
\medskip

Throughout this paper, $z=(z^1,\ldots,z^n)$ will be a holomorphic local coordinate system for the fibers $D_y\subset\CC^n$.
For the base space $S\subset\CC$, we will always use the standard Euclidean coordinate, denoted by $s$.
We will use small Greek letters, $\alpha,\beta,\dots=1,\dots,n$ for indices on $z$ unless otherwise specified. 
For a properly differentiable function $f$ on the total space $D\subset\CC^n\times\CC$, we denote by
\begin{equation} \label{E:convention}
f_\alpha
    =\pd{f}{z^\alpha},\;\;
f_{\bar\beta}
    =\pd{f}{z^{\bar\beta}},\;\;
   \;\;\text{and}\;\;
f_s
    =\pd{f}{s},\;\;
f_{\bar s}
    =\pd{f}{\bar s},
\end{equation}
where $z^{\bar\beta}$ mean $\ov{z^\beta}$.
We will always use the Einstein convention and the same letter $``C"$ to denote a generic constant, which may change from one line to another, but it is independent of the pertinent parameters involved.
\medskip

\noindent\textbf{Acknowledgement.}
The first named author is grateful to R. Berman for suggesting this problem and sharing his ideas.
He was supported by the National Research Foundation
(NRF) of Korea grant funded by the Korea government 
(No. 2018R1C1B3005963).
The work of second named author was supported by IBS-R003-D1. 

\section{Preliminaries}

A compact K\"ahler manifold with negative first Chern class admits an unique \ke metric with negative constant Ricci curvature by Aubin \cite{Aubin} and Yau \cite{Yau} using the continuity method.
Later, Cao \cite{Cao} gave another proof of the existence of K\"{a}hler-Einstein metric using the K\"{a}hler-Ricci flow.

On the other hand, these results can be generalized to non-compact complete K\"{a}hler manifolds which admit properties of bounded geometry due to Cheng-Yau and Chau.
In this section, we recaptulate their results (for the details, see \cite{Chau,Cheng_Yau}).

\subsection{Elliptic \ma equation and K\"{a}hler-Einstein metric}
The existence of K\"{a}hler-Einstein metric comes from the solvability of the complex Monge-Amp\`ere equation. For that purpose, Cheng and Yau introduced the notion of the bounded geometry of non-compact complete K\"{a}hler manifold.

\begin{definition} [Bounded geometry]\label{D:bounded geometry}
Let $(M,\omega)$ be a complete K\"{a}hler manifold of dimension $n$. 
We say that $(M,\omega)$ has \emph{bounded geometry of order $k$} if for each $p\in M$ there exists a holomorphic chart $(U_p,\xi_p)$ centered at $p$ satisfying following conditions:
\begin{itemize}
\item [(\romannumeral1)] There exist constant $r>0$, independent of $p$ satisfying
	\begin{equation*}
	\mathbb{B}_{r}(0)\subset V_p:=\xi_p(U_p)\subset \mathbb{C}^n,
	\end{equation*}
	where $\mathbb{B}_{r}(0)$ denotes the ball of radius $r$ centered at $0$ in $\mathbb{C}^n$.
\item [(\romannumeral2)] There exists a constant $C>0$ independent of $p$ satisfying
	\begin{equation*}
	\frac{1}{C}\paren{\delta_{\alpha\bar\beta}}\leq\paren{g_{\alpha\bar\beta}}\leq C\paren{\delta_{\alpha\bar\beta}},
	\end{equation*}
	where $\omega=i g_{\alpha\bar\beta}d\xi^{\alpha}\wedge d\ov{\xi^{\beta}}$ for the coordinates $\xi_p=(\xi^1,\ldots,\xi^n)$.
\item [(\romannumeral3)] For any $l\leq k$, there exist constants $C_{l}>0$ independent of $p$ satisfying
	\begin{equation*}
	\norm{g_{\alpha\bar\beta}}_{C^{l}({V_p})}\leq C_{l}.
	\end{equation*}
\end{itemize}
\end{definition}

Suppose that a complete K\"{a}hler manifold $(M,\omega)$ has bounded geometry of order $k$.
Let $\{(U_p,\xi_p)\}$ be a family of holomorphic charts covering $M$ and satisfying the conditions in Definition \ref{D:bounded geometry}.
For any functions $u\in C^{\infty}(M)$, we define a norm by
$$
\norm{u}_{k+\epsilon}:=\sup_{p\in M}\{\norm{u\circ\xi_p^{-1}}_{C^{k+\epsilon}({V_p})}\},
$$
where $\norm{\cdot}_{C^{k+\epsilon}({V_p})}$ is the standard elliptic H\"{o}lder norm on $V_p:=\xi_p(U_p)\subset\mathbb{C}^n$.
We denote the Banach completion of the space $\{u\in C^{\infty}(M): \norm{u}_{k+\epsilon}<\infty\}$ by $\tilde{C}^{k+\epsilon}(M)$.

Now we can state the following theorem due to Cheng and Yau.

\begin{theorem} [Theorem 4.4 in \cite{Cheng_Yau}] \label{T:CY1}
Suppose $(M,\omega)$ is a complete K\"{a}hler manifold with bounded geometry of order $k\ge5$. 
Then, for any $K>0$ and $F\in\tilde{C}^{k-2+\varepsilon}(M)$, there exists a unique $\psi\in\tilde{C}^{k+\varepsilon}(M)$ satisfying the following conditions:
\begin{align} 
(\omega+i\partial\ov\partial\psi)^n
=e^{K\psi+F}\omega^n, \label{E:Elliptic Monge-Ampere}\\
\frac{1}{C}\omega\leq\omega+i\partial\ov\partial\psi\leq C\omega. \label{E:Quasi isometry}
\end{align}
Moreover, if all the data are analytic, the solution is also analytic.
\end{theorem}

\begin{remark}
The equation \eqref{E:Elliptic Monge-Ampere} is called the \emph{elliptic complex \ma equation}.
The inequality \eqref{E:Quasi isometry} implies that $(M,\omega+i\partial\ov\partial \psi)$ also has bounded geometry of order $k$ (see Proposition 1.4 in \cite{Cheng_Yau}).
\end{remark}

We further assume that the K\"{a}hler form $\omega$ satisfies the following condition:
\begin{equation} \label{E:CCohomology} 
\ric(\omega)+K\omega=i\partial\ov\partial F,
\end{equation}
for some constant $K>0$ and function $F\in\tilde{C}^{k-2+\varepsilon}(M)$.
Consider the K\"{a}helr metric $\omega_{KE}:=\omega+i\partial\ov\partial\psi$, where $\psi$ is the solution of the \ma equation \eqref{E:Elliptic Monge-Ampere} in Theorem \ref{T:CY1}.
Then we have the following

\begin{theorem} [Cheng-Yau \cite{Cheng_Yau}] \label{T:CY2}
The K\"{a}hler metric $\omega_{KE}:=\omega+i\partial\ov\partial\psi$ is the unique complete K\"{a}hler-Einstein metric of $M$ satisfying $\ric(\omega_{KE})=-K\omega_{KE}$.
\end{theorem}

\subsection{Parabolic \ma equation and K\"{a}hler-Ricci flow}

There is an alternative proof of Theorem \ref{T:CY2} using Hamilton's Ricci flow due to Chau \cite{Chau}.
This flow is called the \emph{K\"{a}hler-Ricci flow} since it preserves the K\"{a}hler-ness along the flow.

One of the advantages of the K\"{a}hler-Ricci flow approach is that one can prove the existence of K\"{a}hler-Einstein metric under weaker assumptions.
More precisely, Chau proved the following parabolic version of Theorem \ref{T:CY2}.

\begin{theorem} [Chau \cite{Chau}] \label{T:Chau2}
Let $(M,\omega)$ be a complete K\"{a}hler manifold with bounded curvature.
Suppose that there exists a smooth bounded function $F$ satisfying 
\begin{equation*}
\ric(\omega)+K\omega=i\partial\ov\partial F.
\end{equation*}
Then there exist a time family of K\"{a}hler metrics $\omega(t)$ for all $t>0$ satisfying 
\begin{equation} \label{E:Kahler-Ricci flow'}
\begin{aligned}
\frac{\partial}{\partial t}\omega(t)
&=
-\ric(\omega(t))-K\omega(t),\medskip\\
\omega(0)&=\omega.
\end{aligned}
\end{equation}
Moreover, $\omega(t)$ converges to the unique complete \ke metric $\omega_{KE}$.
\end{theorem}

The equation \eqref{E:Kahler-Ricci flow'} is called the (normalized) {\it K\"ahler-Ricci flow}.
Note that here, we assumed neither the conditions of bounded geometry for $\omega$ nor $F\in\tilde{C}^{k-2+\varepsilon}(M)$.
But one can always find such metrics using the short time existence of K\"ahler-Ricci flow due to Shi \cite{Shi} so that the K\"{a}hler-Einstein metric exists by Theorem \ref{T:CY2}. 

To prove the long time existence of K\"ahler-Ricci flow, Chau considered the functions $\varphi\in C^{\infty}(M\times[0,\infty))$ such that $
\omega(t):=\omega+i\partial\ov\partial\varphi(t)
$ satisfies the equation \eqref{E:Kahler-Ricci flow'}.
Then the problem is reduced to the solvability of the following \emph{parabolic complex \ma equation}.

\begin{theorem} [Theorem 1.1 in \cite{Chau}] \label{T:Chau}
There exists a solution $\varphi\in\tilde{C}^{k+\epsilon,\frac{k+\epsilon}{2}}(M\times[0,\infty))$ of the following equation:
\begin{equation}  \label{E:Parabolic Monge-Ampere'}
\left\{ \begin{array}{l}
\frac{\partial}{\partial t}\varphi
=
\log\frac{(\omega+i\partial\ov\partial\varphi)^n}{\omega^n}-K\varphi-F,\medskip\\
\varphi|_{t=0}=0.
\end{array}\right.
\end{equation}
Moreover, $\varphi(t)$ converges to the function $\psi$ in $\tilde{C}^{k+\varepsilon}(M)$ as $t\rightarrow\infty$, where $\psi$ is the unique solution of the equation \eqref{E:Elliptic Monge-Ampere} in Theorem \ref{T:CY1}. 
\end{theorem}

Here, the space $\tilde{C}^{k+\epsilon,\frac{k+\epsilon}{2}}(M\times[0,T))$ is the Banach completion of the parabolic H\"{o}lder space $\{u\in C^{\infty}(M\times[0,T)): \norm{u}_{k+\epsilon,\frac{k+\epsilon}{2}}<\infty\}$
with the norm
\begin{equation*}
\norm{u}_{k+\epsilon,\frac{k+\epsilon}{2}}
:=
\sup_{p\in M}
\left
	\{\norm{u\circ\xi_p^{-1}}_{C^{k+\epsilon,\frac{k+\epsilon}{2}}(V_p\times[0,T))}
\right\},
\end{equation*}
where $\norm{\cdot}_{C^{k+\epsilon,\frac{k+\epsilon}{2}}(V_p\times[0,T))}$ is the standard parabolic H\"{o}lder norm on $V_p\times[0,T)$.

\section{Fiberwise K\"{a}hler-Ricci flow}

In this section, we discuss the variation of the K\"ahler-Ricci flows on a holomorphic family of bounded strongly pseudoconex domains, which gives the fiberwise K\"ahler-Ricci flow.
Moreover, we will prove that the fiberwise K\"ahler-Ricci flow converges the fiberwise K\"ahler-Einstein metric.

\subsection{Construction of the reference form}\label{S:initial}

First recall the setting in Introduction :
Let $D$ be a domain in $\CC^{n+1}$ and $S:=\pi(D)\subset\CC$. 
Suppose that $D$ is a \emph{holomorphic family of bounded strongly pseudoconvex domains}, i.e., it satisfies the following:
\begin{itemize}

\item[(\romannumeral1)] $\pi^{-1}(S)\cap\partial D$ is smooth and $\pi\vert_{\partial D}:\pi^{-1}(S)\cap\partial D\rightarrow S$ is a submersion. 

\item[(\romannumeral2)] For $y\in S$, all fibers $D_y:=\pi^{-1}(y)\cap D$ are smoothly bounded strongly pseudoconvex domains in $\CC^{n}$.

\end{itemize}

Note that the Condition (\romannumeral1) implies that all fibers are diffeomorphic by Ehresmann's fibration theorem (cf. \cite{Saeki}).
Together with the Condition (\romannumeral2), there exists a defining function $r$ of $D$ such that $r\vert_{\ov D_y}$ is a strictly plurisubharmonic function on $\ov D_y$.
Define a $d$-closed smooth $(1,1)$-form on the total space $D$ by
$$
\omega:=i\partial\bar\partial (-\log(-r)),
$$
where $\partial$ and $\bar\partial$ are the operators of the total space $\CC^{n+1}$.
Then one can check that $(D_y,\omega_y)$ is a complete K\"{a}hler manifold with bounded geometry of infinite order (for the details, see \cite{Cheng_Yau}).
However, there is no information about the positivity of the reference form $\omega$ along the base direction.
The following theorem says that positivity of $\omega$ on $D$ is guaranteed by the pseudoconvexity of $D$ in $\CC^{n+1}$. 

\begin{proposition} \label{P:Pseudoconvex}
If $D$ is pseudoconvex on $\CC^{n+1}$, then there exists a defining function $r$ of $D$ such that $\omega:=i\partial\bar\partial (-\log(-r))$ satisfies the following conditions
\begin{itemize}
\item $\omega_y:=\omega\vert_{D_y}$ is complete K\"{a}hler form on each fiber $D_y$.
\item $\omega\geq0$ on $D$, and $\omega$ is strictly positive at least one point on each fiber $D_y$.
\end{itemize}

\end{proposition}
\begin{proof}
Note that $D$ is a holomorphic family of bounded strongly pseudoconvex domains, which is pseudoconvex in $\CC^{n+1}$.
Then there exists a smooth plurisubharmonic defining function $\tilde{r}$ of $D$ such that $\tilde{r}\vert_{\ov D_y}$ is a strictly plurisubharmonic function on $\ov D_y\cap U$, where $U$ is a neighborhood of $\pi^{-1}(S)\cap\partial D$.
Let $\epsilon_1, \epsilon_2$ be negative constants satisfying $\{x\in U: \epsilon_1<\hat{r}(x)<\epsilon_2<0\}\subset\subset U\cap D$.
Choose $\chi\in C^{\infty}({\mathbb{R}})$ such that $\chi$ is negative constant for $t\leq\epsilon_1$, $\chi(t)=t$ for $t\geq\epsilon_2$, and $\chi',\chi''>0$ for $\epsilon_1<t<\epsilon_2$. 
Then for a suitable cutoff function $\lambda$, $r:=\chi\circ\tilde{r}+\lambda|z|^2$ is a smooth defining function of $D$ satisfying the conditions in the statement of the proposition.
\end{proof}

Since $\omega_y>0$ for each fiber, the {\it relative curvature form} of $\omega$ can be defined by
$$
\Theta_\omega:=i\partial\bar\partial\log(\omega^n\wedge dV_s),
$$
where $dV_s:=ids\wedge d\ov s$ is the volume form on the base space $\CC$ (cf. \cite{Choi_ych}).
In fact, this is the curvature form of the relative canonical line bundle.
The following proposition will be used later to prove Theorem \ref{T:Relative Kahler-Ricci flow} and Theorem \ref{berman_flow}.

\begin{proposition}\label{P:initial_coho}
There exists a bounded smooth function $F$ on $D$ satisfying
\begin{equation}\label{E:Cohomology}
-\Theta_\omega+(n+1)\omega=i\partial\bar\partial F.
\end{equation}
Moreover, $F$ is smoothly extended up to $\partial D$.
\end{proposition}

\begin{proof}
Let $(z^1,\ldots,z^n,s)$ be the Euclidean coordinate for $\CC^{n+1}$.
Then we have
$$
\Theta_\omega=\ddc\log\det(g_{\alpha\bar\beta}),
$$
where $g:=-\log(-r)$ is a function on $D$.
The computations in \cite{Cheng_Yau} shows that
\begin{equation*}
\det(g_{\alpha\bar\beta})
=
\paren{\frac{1}{-r}}^{n+1}\det(r_{\alpha\bar\beta})\paren{-r+\abs{\partial r}^2},
\end{equation*}
where $\abs{\partial r}^2:=r^{\alpha\bar\beta}r_\alpha r_{\bar\beta}$ with $r^{\alpha\bar\beta}=(r_{\alpha\bar\beta})^{-1}$.
It follows that
\begin{align*}
\Theta_\omega
=
\ddc\log\det(g_{\alpha\bar\beta})
=
(n+1)\omega
+
\ddc\log\paren{\det(r_{\alpha\bar\beta})\paren{-r+\abs{\partial r}^2}}.
\end{align*}
If we define the function $F:D\rightarrow\RR$ by
\begin{equation*}
F:=-\log\paren{\det(r_{\alpha\bar\beta})\paren{-r+\abs{\partial r}^2}},
\end{equation*}
then $F$ is a bounded smooth function satisfying the equation \eqref{E:Cohomology}.
Since $r$ is smooth on $\ov D$ and $\abs{\partial r}\neq0$ on $\partial D$, the second assertion follows.
\end{proof}

\subsection{Fiberwise K\"{a}hler-Ricci flow}
Note that $\Theta_\omega|_{D_y}=-\ric(\omega_y)$.
Restricting the equation \eqref{E:Cohomology} to the fiber $D_y$, we have
$$
\ric(\omega_y)+(n+1)\omega_y=i\partial\bar\partial F_y,
$$
where $F_y:=F|_{D_y}$.
Therefore, Theorem \ref{T:Chau} implies that for $y\in S$, there exists a solution $\varphi_y$ on $D_y\times[0,\infty)$ of the parabolic \ma equation:
\begin{equation}  \label{E:Fiberwise Parabolic Monge-Ampere}
\begin{aligned}
\frac{\partial}{\partial t}\varphi_y
&=
\log\frac{(\omega_y+i\partial\bar\partial\varphi_y)^n}{\omega_y^n}-(n+1)\varphi_y-F_y,
\\
\varphi_y|_{t=0}
&=0
\end{aligned}
\end{equation}
Hence $\omega_y(t):=\omega_y+i\partial\bar\partial\varphi_y(t)$ 
is the solution of the (normalized) {\it K\"{a}hler-Ricci flow}:
\begin{equation} \label{E:Fiberwise Kahler-Ricci flow}
\begin{aligned}
\pd{}{t}\omega_y(t)
&
=-\ric(\omega_y(t))-(n+1)\omega_y(t),
\\
\omega_y(0)
&=
\omega_y.
\end{aligned}
\end{equation}
The following proposition yields that the solution $\varphi_y(t)$ of the equations \eqref{E:Fiberwise Parabolic Monge-Ampere} vary smoothly along the base direction $s$.
\begin{proposition}\label{P:Regularity}
For $t\in[0,\infty)$, the function $\varphi(t)$ given by 
$$
\varphi(x;t):=\varphi_y(x;t)
$$
where $y=\pi(x)$ and $x\in D$, is smooth on the total space $D$.
\end{proposition}
\begin{proof}
For a fixed point $y_0\in S$, denote by $\Omega:=D_{y_0}$. 
Ehressmann's fibration theorem implies that there exists a fiber-preserving diffeomorphism $\Phi:D\rightarrow\Omega\times S$ which is smoothly extended up to the boundary $\pi^{-1}(S)\cap\partial{D}$. 
Hence for $y\in S$, all Banach spaces $\tilde{C}^{k+\epsilon,\frac{k+\epsilon}{2}}(D_y\times[0,T))$ can be identified with the space $\tilde{C}^{k+\epsilon,\frac{k+\epsilon}{2}}(\Omega\times[0,T))$.

Now we define the following {\it parabolic \ma operator}
$$
\mathcal{M}:U\times\tilde{C}^{k+\epsilon,\frac{k+\epsilon}{2}}(\Omega\times[0,T))\rightarrow\tilde{C}^{k-2+\epsilon,\frac{k-2+\epsilon}{2}}(\Omega\times[0,T))
$$
by
$$
\mathcal{M}(y,\phi)=\frac{\partial}{\partial t}\phi-\log\frac{(\omega_y+i\partial\bar\partial\phi_y)^n}{\omega_y^n}+(n+1)\phi+F_y.
$$
By Theorem \ref{T:Chau}, there exists $\varphi_{y_0}\in\tilde{C}^{k+\epsilon,\frac{k+\epsilon}{2}}(\Omega\times[0,T))$ such that
$$
\mathcal{M}(y_0,\varphi_{y_0})=0.
$$ 
Then, the partial Fr\'{e}chet derivative of $\mathcal{M}$ at the point $(y_0,\varphi_{y_0})$ is an operator
$$
D_2\mathcal{M}(y_0,\varphi_{y_0}):\tilde{C}^{k+\epsilon,\frac{k+\epsilon}{2}}(\Omega\times[0,T))\rightarrow\tilde{C}^{k-2+\epsilon,\frac{k-2+\epsilon}{2}}(\Omega\times[0,T))
$$
which is defined by for any $\phi\in\tilde{C}^{k+\epsilon,\frac{k+\epsilon}{2}}(\Omega\times[0,T))$,
$$
D_2\mathcal{M}(y_0,\varphi_{y_0})(\phi)=\big(\frac{\partial}{\partial t}-\Delta_t+(n+1)\cdot id\big)\phi,
$$
where $\Delta_t$ is the Laplacian with respect to $\omega_{y_0}(t)=\omega_{y_0}+i\partial\bar\partial\varphi_{y_0}(t)$.

Using a version of maximum principle, we can show that $D_2\mathcal{M}(y_0,\varphi_{y_0})$ is a Banach space isomorphism between $\tilde{C}^{k+\epsilon,\frac{k+\epsilon}{2}}(\Omega\times[0,T))$ and $\tilde{C}^{k-2+\epsilon,\frac{k-2+\epsilon}{2}}(\Omega\times[0,T))$ (for details, see the proof of Claim 1 of Lemma 2.2 in \cite{Chau_Tam}). 
Hence the Implicit Function Theorem implies that there exists a Fr\'{e}chet differentiable function $\mu:U\rightarrow\tilde{C}^{k+\epsilon,\frac{k+\epsilon}{2}}(\Omega\times[0,T))$ such that
$$
\mathcal{M}(y,\mu(y))=0.
$$
The uniqueness of the solution implies $\mu(y)=\varphi_y$ so that $\varphi_y\in\tilde{C}^{k+\epsilon,\frac{k+\epsilon}{2}}(\Omega\times[0,T))$.
Since the Monge-Amp\`{e}re operator $\mathcal{M}$ is smooth, the implicit function theorem implies that $\varphi_y$ vary smoothly along the base direction $s$. 
\end{proof}

\begin{remark}\label{R:Boundedness of u}
The proof of Proposition \ref{P:Regularity} implies that for any $l_1,l_2=0,1,2,\dots$,
\begin{equation*}
\paren{\frac{\partial}{\partial s}}^{l_1}\paren{\frac{\partial}{\partial\ov s}}^{l_2}\varphi\Big|_{D_y}\in\tilde C^{k+\epsilon,\frac{k+\epsilon}{2}}(D_y\times[0,\infty)).
\end{equation*}

\end{remark}

By Proposition \ref{P:Regularity}, we can define $d$-closed smooth real $(1,1)$-forms $\omega(t)$ on the total space $D$ by
\begin{equation} \label{E:defomegat}
\omega(t):=\omega+i\partial\bar\partial \varphi(t).
\end{equation}
Since $\omega_y(t):=\omega(t)|_{D_y}$ is K\"{a}hler on each fiber $D_y$,
one can consider relative curvature form $\Theta_{\omega(t)}$ of $\omega(t)$ on $D$, given by
$$ 
\Theta_{\omega(t)}=i\partial\bar\partial\log(\omega(t)^n\wedge dV_s).
$$

\begin{theorem}\label{T:Relative Kahler-Ricci flow}
The form $\omega(t)$ on $D$ satisfies the following equation:
\begin{equation} \label{E:Relative Kahler-Ricci flow}
\begin{aligned}
\frac{\partial}{\partial t}\omega(t)
&=
\Theta_{\omega(t)}-(n+1)\omega(t),\medskip\\
\omega(0)
&=
\omega.
\end{aligned}
\end{equation}
\end{theorem}

\begin{proof}
It follows from \eqref{E:Fiberwise Parabolic Monge-Ampere} that
\begin{align*}
\frac{\partial}{\partial t}\varphi
=
\log\frac{(\omega+i\partial\bar\partial\varphi)^n\wedge dV_s}{\omega^n\wedge dV_s}
-(n+1)\varphi-F.
\end{align*}
Taking $\ddc$, we have
\begin{align*}
\ddc\paren{\frac{\partial}{\partial t}\varphi}
=
\Theta_{\omega(t)}-\Theta_\omega
-(n+1)\ddc\varphi-\ddc F.
\end{align*}
Since $\omega$ does not depend on $t$, \eqref{E:Cohomology} implies that
\begin{align*}
\pd{}{t}\paren{\omega+\ddc\varphi}
=
\Theta_{\omega(t)}
-(n+1)\omega.
\end{align*}
This completes the proof.
\end{proof}

\begin{remark}
We will call the equation \eqref{E:Relative Kahler-Ricci flow} the {\it fiberwise K\"{a}hler-Ricci flow} on $D$, since the restriction of it to the fiber $D_y$ is equal to the equation \eqref{E:Fiberwise Kahler-Ricci flow}.
This flow was first introduced by Berman in \cite{Berman} with the name ``{\it relative K\"{a}hler-Ricci flow}".
\end{remark}

\subsection{Fiberwise \ke metric}
On the other hand,
Theorem \ref{T:Chau} implies that for all fibers $D_y$, the solution $\varphi_y(t)$ of the parabolic \ma equation \eqref{E:Fiberwise Parabolic Monge-Ampere} converges to the solution
$\psi_y$ of the elliptic \ma equation:
\begin{equation} \label{E:Fiberwise Elliptic Monge-Ampere}
\begin{aligned} 
(\omega_y+i\partial\bar\partial\psi_y)^n
&=e^{(n+1)\psi_y+F_y}\omega_y^n,
\\
\frac{1}{C}\omega_y\leq\omega_y&+i\partial\ov\partial\psi_y\leq C\omega_y.
\end{aligned}
\end{equation}
By the uniqueness of the \ke metric, we have
$$
\omega_y+i\partial\bar\partial\psi_y=\omega_y^{KE},
$$
where $\omega_y^{KE}$ is the unique \ke metric with Ricci curvature $-(n+1)$.
As in Proposition \ref{P:Regularity}, the implicit function theorem for the elliptic \ma operator implies the following
\begin{proposition}[cf. Section 3 in \cite{Choi1}]
The function $\psi:D\rightarrow\RR$, defined by
$$
\psi(x):=\psi_y(x)
$$
where $y=\pi(x)$, is smooth on the total space $D$.
\end{proposition}

Define a $d$-closed smooth $(1,1)$-form $\rho$ on the total space $D$ by
$$
\rho:=\omega+i\partial\ov\partial\psi.
$$
The form $\rho$ is called the {\it fiberwise \ke metric}, since $\rho|_{D_y}=\omega_y^{KE}$.

\begin{theorem}\label{T:convergence of fiberwise}
The solution of fiberwise K\"ahler-Ricci flow $\omega(t)$ locally uniformly converges to the fiberwise K\"ahler metric $\rho$ on $D$ as $t\rightarrow\infty$.
More precisely,  we have that $\varphi(t)\rightarrow\psi$ in $C^{\infty}_{loc}(D)$.
\end{theorem}

\begin{proof}
It is enough to show that $\varphi(t)$ smoothly converges to $\psi$ on any compact subset of $D$.
More precisely, we will show that for each point $x\in D$, there exists a neighborhood $U$ of $x$ in $D$ such that
$$
\norm{\varphi(t)-\psi}_{C^k(U)}\rightarrow 0
$$
as $t\rightarrow\infty$, for all $k\ge0$.
Before going to the proof, note that for $y\in S$, we already know that as $t\rightarrow\infty$,
\begin{equation}\label{E:fiber_estimate}
\norm{\varphi(t)\vert_{D_y}-\psi\vert_{D_y}}_{\tilde C^{k,\epsilon}(D_y)}\rightarrow0.
\end{equation}

First consider the $C^0$-convergence.
Differentiating \eqref{E:Fiberwise Parabolic Monge-Ampere} with respect to $t$, we get
\begin{equation}
\begin{aligned}
\frac{\partial}{\partial t}\dot{\varphi_y}
&=
\Delta_t\dot{\varphi_y}-(n+1)\dot{\varphi_y},
\\
\dot{\varphi_y}|_{t=0}
&=
F_y.
\end{aligned}
\end{equation}
It follows that
$$
\frac{\partial}{\partial t}(e^{(n+1)t}\dot{\varphi_y})=\Delta_t(e^{(n+1)t}\dot{\varphi_y}).
$$
A maximum principle implies that
$$
\abs{e^{(n+1)t}\dot{\varphi_y}}\leq \sup_{D_y}\abs{F_y}\le C
$$
for some uniform constant $C>0$, independent of $y$.
For $0<t'<t''$, we have
\begin{align*}
\abs{\varphi_y(x,t')-\varphi_y(x,t'')}
&\le
\abs{\int_{t'}^{t''}\dot\varphi_y(x,u)du}
\le
\int_{t'}^{t''}\abs{\dot\varphi_y(x,u)}du \\
&\le
\int_{t'}^{t''}Ce^{-(n+1)u}du 
\le
C\paren{e^{-(n+1)t'}-e^{-(n+1)t''}}.
\end{align*}
By \eqref{E:fiber_estimate}, this implies that 
\begin{equation*}
\norm{\varphi(t)-\psi}_{C^0(D)}\le Ce^{-(n+1)t}.
\end{equation*}

Now we consider the $C^k$-convergence for any fixed $k\in\NN$.
For each $l_1,l_2\in\NN$ with $l_1+l_2\le k$, the proof of Proposition \ref{P:Regularity} implies that
\begin{equation*}
U\ni y \rightarrow 
D_s^{l_1,l_2}\varphi_y(t):=\Big(\frac{\partial}{\partial s}\Big)^{l_1}\Big(\frac{\partial}{\partial\ov s}\Big)^{l_2}\varphi(t)\Big|_{D_y}\in\tilde C^{k+\epsilon,\frac{k+\epsilon}{2}}(D_y\times[0,\infty)).
\end{equation*}
is smooth where $U$ is a neighborhood of $y$.
Hence there exists a uniform constant $C$ which depends only on $l_1,l_2,k$ such that
\begin{equation*}
\sup_{y\in U}
\sum_{l_1+l_2\le k}\norm{D_s^{l_1,l_2}\varphi_y(t)}_{\tilde C^{k+\epsilon,\frac{k+\epsilon}{2}}(D_y\times[0,\infty))}
<C.
\end{equation*}
This implies that there exists a neighborhood $V$ of $x$ in $D$ and an uniform constant $C$ which does not depend on $t$ such that
\begin{equation*}
\norm{\varphi(t)}_{C^k(V)}<C
\end{equation*}
where $C^k(V)$-norm means the usual $C^k$-norm on $V\subset\CC^{n+1}$.
Therefore, the proof is completed by the Arzela-Ascoli theorem and the uniqueness of limit (for the details, see \cite{Cao, Song_Weinkove}).
\end{proof}

Theorem \ref{T:Relative Kahler-Ricci flow} and Theorem \ref{T:convergence of fiberwise} imply the following

\begin{corollary} [cf. Remark 3.4 in \cite{Choi_Yoo}]
The fiberwise K\"ahler-Einstein metric $\rho$ satisfies the equation
\begin{equation}
\Theta_{\rho}=(n+1)\rho,
\end{equation}
where $\Theta_\rho$ is the relative curvature form of $\rho$.
\end{corollary}

\section{Geodesic curvature of the fiberwise K\"{a}hler-Ricci flow}

In this section, we introduce the horizontal lift, which is developed by Siu and Schumacher (cf. \cite{Siu, Schumacher}), and the geodesic curvature which measures the positivity of a fiberwise K\"ahler form. We also discuss Berman's parabolic PDE which the geodesic curvature of the fiberwise K\"ahler-Ricci flow satisfies.

\subsection{Horizontal lift and Geodesic curvature}
Let $D$ be a domain in $\CC^{n+1}$ such that every fiber $D_y$ is a domain in $\CC^n$ for $y\in S:=\pi(D)$. Denote by $v:=\frac{\partial}{\partial s}\in T'_yS$ the coordinate vector field in the base.

\begin{definition}
Let $\tau$ be a $d$-closed smooth real $(1,1)$-form on $D$ whose restriction to the fibers $\tau|_{D_y}$ is positive definite.
\begin{itemize}
\item[(1)] A vector field $v_\tau$ of type $(1,0)$ is called the \emph{horizontal lift} along $D_y$ of $v$ with respect to $\tau$ if $v_\tau$ satisfies the following:
\begin{itemize}
\item [(\romannumeral1)]$\inner{v_\tau,w}_\tau=0$ for all $w\in{T'D_y}$,
\item [(\romannumeral2)]$d\pi(v_\tau)=v$.
\end{itemize}
\item[(2)] The \emph{geodesic curvature} $c(\tau)(v)$ of $\tau$ along $v$ is defined by the norm of $v_\tau$ with respect to the sesquilinear form $\inner{\cdot,\cdot}_\tau$ induced by $\tau$, namely,
\begin{equation*}
c(\tau):=c(\tau)(v)=\inner{v_\tau,v_\tau}_\tau.
\end{equation*}
\end{itemize}
\end{definition}

\begin{remark}\label{R:geo_curv}
We have the following remarks.
\begin{itemize}
\item[(1)]
Under a local coordinate system $(z^1,\ldots,z^n,s)$, $\tau$ can be written as
$$
\tau
=
i\paren{
\tau_{s\bar s}ds\wedge d\ov s
+
\tau_{\alpha\bar s}dz^{\alpha}\wedge d\ov{s}
+
\tau_{s\bar s}ds\wedge dz^{\bar\beta}
+
\tau_{\alpha\bar\beta}dz^{\alpha}\wedge dz^{\bar\beta}
}.
$$
Then the horizontal lift $v_\tau$ and the geodesic curvature $c(\tau)$ are given by
$$
v_{\tau}=\frac{\partial}{\partial s}-\tau_{s\bar\beta}\tau^{\bar\beta\alpha}\frac{\partial}{\partial z^{\alpha}}
\;\;\;\text{and}\;\;\;
c(\tau)=
\tau_{s\bar{s}}-\tau_{s\bar\beta}\tau^{\bar\beta\alpha}\tau_{\alpha\bar{s}},
$$
where $(\tau^{\bar\beta\alpha})$ is the inverse matrix of $(\tau_{\alpha\ov\beta})$.
\item[(2)] The following identity is well-known and important (cf. \cite{Schumacher}):
\begin{equation*}
\frac{\tau^{n+1}}{(n+1)!}=c(\tau)\cdot\frac{\tau^n}{n!}\wedge ids\wedge d\bar{s}.
\end{equation*}
Since $\tau|_{D_y}>0$, this implies that $c(\tau)\geq0$ if and only if $\tau$ is a semi-positive real $(1,1)$-form on $D$. Furthermore, $c(\tau)>0$ if and only if $\tau$ is positive.
\end{itemize}
\end{remark}

\subsection{Berman's parabolic PDE}
Let $D$ be a holomorphic family of bounded strongly pseudoconvex domains.
Then the geodesic curvature $c(\omega(t))$ satisfies a certain parabolic PDE, which was first computed by Berman for a family of canonically polarized compact K\"ahler manifolds.
The following theorem is essentially the same with Berman's one, but we will give a precise proof for the reader's convenience.

\begin{theorem} [cf. Theorem 4.7 in \cite{Berman}]\label{berman_flow}
For each fiber $D_y$, $c(\omega(t))|_{D_y}$ evolves by
\begin{equation}\label{E:Evolution}
\paren{\frac{\partial}{\partial t}-\Delta_t}c(\omega(t))+(n+1)c(\omega(t))
=
\norm{\ov{\partial}v_{\omega(t)}}^2,
\end{equation}
where $\Delta_t$ is the Laplace-Beltrami operator of the K\"ahler metric $\omega_y(t):=\omega(t)|_{D_y}$.
\end{theorem}

\begin{proof}
Note that 
$\omega(t)=i\partial\ov\partial g(t)$ on $D$,
where $g(t):=-\log(-r)+\varphi(t)$.
During this proof, for simplicity, we will omit $t$ for the function $g(t)=:g$.
Then $\omega(t)$ can be written as follows:
$$
\omega(t)
=
i
\paren{
	g_{s\bar{s}}ds\wedge d\ov{s}
	+
	g_{\alpha\bar{s}}dz^{\alpha}\wedge d\ov{s}
	+
	g_{s\bar\beta}ds\wedge dz^{\bar\beta}
	+
	g_{\alpha\bar{\beta}}dz^{\alpha}\wedge dz^{\bar\beta}
}.
$$
As we saw in Remark \ref{R:geo_curv}, the geodesic curvature is given by
\begin{equation*}
c(\omega(t))
=
g_{s\bar{s}}
-
g_{s\bar\beta}g^{\bar\beta\alpha}g_{\alpha\bar{s}}.
\end{equation*}
Thus we have
\begin{equation*}
\pd{}{t}c(\omega(t))
=
\paren{\pd{}{t}g}_{s\bar{s}}
-
\paren{\pd{}{t}g}_{s\bar\beta}g^{\bar\beta\alpha}g_{\alpha\bar{s}}
-
g_{s\ov{\beta}}\paren{\pd{}{t}g^{\bar\beta\alpha}}g_{\alpha\bar{s}}
-
g_{s\bar\beta}g^{\bar\beta\alpha}\paren{\pd{}{t}g}_{\alpha\bar s}.
\end{equation*}
On the other hand,
\begin{align*}
\Delta_t c(\omega(t))
=&
\Delta_t g_{s\ov{s}}
-
\Delta_t(g_{s\bar\beta}g^{\bar\beta\alpha})g_{\alpha\bar s}
-
g^{\bar\delta\gamma}
(g_{s\bar\beta}g^{\bar\beta\alpha})_\gamma(g_{\alpha\bar s})_{\bar\delta}
\\
&-
g^{\bar\delta\gamma}
(g_{s\bar\beta}g^{\bar\beta\alpha})_{\bar\delta}(g_{\alpha\bar s})_\gamma
-
(g_{s\bar\beta}g^{\bar\beta\alpha})\Delta_t g_{\alpha\bar s}\\
=&I_0-I_1-I_2-I_3-I_4.
\end{align*}
Notice that
$
(\log\det(g_{\alpha\ov{\beta}}))_{\ov{s}}
=
g^{\bar\delta\gamma}(g_{\bar s})_{\gamma\bar\delta}=\Delta_tg_{\ov{s}}.
$
This implies that
\begin{align*}
I_0
:=
\Delta_t g_{s\bar s}
&=
g^{\bar\delta\gamma}(g_{s\ov{s}})_{\gamma\bar\delta}
=
g^{\bar\delta\gamma}(g_{\gamma\bar\delta})_{s\bar s}
=
(g^{\bar\delta\gamma}(g_{\gamma\delta})_{\ov{s}})_s
-
(g^{\bar\delta\gamma})_s(g_{\gamma\bar\delta})_{\bar{s}}
\\
&=
(g^{\bar\delta\gamma}(g_{\bar s})_{\gamma\bar\delta})_s
+
g^{\bar\delta\alpha}(g_{\alpha\bar\beta})_s
g^{\bar\beta\gamma}(g_{\gamma\bar\delta})_{\bar s}
\\
&=(\log\det(g_{\alpha\ov{\beta}}))_{s\bar s}
+
g^{\bar\delta\alpha}(g_{\alpha\bar\beta})_s
g^{\bar\beta\gamma}(g_{\gamma\bar\delta})_{\bar s}
\\
&=\paren{\frac{\partial}{\partial t}g}_{s\bar s}
+
(n+1)g_{s\bar s}
+
g^{\bar\delta\alpha}(g_{\alpha\bar\beta})_s
g^{\bar\beta\gamma}(g_{\gamma\bar\delta})_{\bar s}.
\end{align*}
In the last equality, we used the fact that $\omega(t)$ satisfies the equation \eqref{E:Relative Kahler-Ricci flow} so that
$$
\paren{\frac{\partial}{\partial t}g}_{s\ov{s}}
=
(\log\det(g_{\alpha\ov{\beta}}))_{s\ov{s}}-(n+1)g_{s\ov{s}}.
$$
From now on, we fix a point and choose a normal coordinate $(z^1,\ldots,z^n)$ such that
\begin{equation*}\label{normal}
\frac{\partial g_{\alpha\bar\beta}}{\partial z^{\gamma}}(x)=0=\frac{\partial g_{\alpha\bar\beta}}{\partial z^{\bar\delta}}(x).
\end{equation*}
Then the term $I_1:=\Delta_t(g_{s\bar\beta}g^{\bar\beta\alpha})g_{\alpha\bar s}$ can be simplified as follows:
\begin{align*}
I_1
=&
g_{s\ov{\alpha}}\Delta_t(g^{\bar\beta\alpha})g_{\gamma\ov{s}}
+
g^{\bar\delta\gamma}(g_{s\ov{\alpha}})_{\bar\delta}(g^{\bar\beta\alpha})_\gamma g_{\gamma\ov{s}}
+
g^{\bar\delta\gamma}(g_{s\bar\beta})_\gamma(g^{\bar\beta\alpha})_{\bar\delta} g_{\alpha\bar s}
+
\Delta_t(g_{s\bar\beta})g^{\bar\beta\alpha}g_{\alpha\bar s}\\
=&
g_{s\bar\beta}\Delta_t(g^{\bar\beta\alpha})g_{\alpha\bar s}
+
g^{\bar\beta\alpha}\Delta_t(g_{s\bar\beta})g_{\alpha\bar s}.
\end{align*}
To compute the first term, note that
\begin{align*}
(g^{\bar\beta\alpha})_{\gamma\bar\delta}
&=
(-g^{\bar\beta\sigma}(g_{\sigma\bar\tau})_{\bar\delta}g^{\bar\tau\alpha})_\gamma
\\
&=
-(g^{\bar\beta\sigma})_\gamma(g_{\sigma\bar\tau})_{\bar\delta}g^{\bar\tau\alpha}
-g^{\bar\beta\sigma}(g_{\sigma\bar\tau})_{\gamma\bar\delta}g^{\bar\tau\alpha}
-g^{\bar\beta\sigma}(g_{\sigma\bar\tau})_{\bar\delta}(g^{\bar\tau\alpha})_\gamma\\
&=-g^{\bar\beta\sigma}(g_{\sigma\bar\tau})_{\gamma\bar\delta}g^{\bar\tau\gamma}.
\end{align*}
This implies that
\begin{align*}
\Delta_t\paren{g^{\bar\beta\alpha}}
=
g^{\bar\delta\gamma}\paren{g^{\bar\beta\alpha}}_{\gamma\bar\delta}
&=
-g^{\bar\beta\sigma}g^{\bar\tau\alpha}
\paren{g^{\bar\delta\gamma}\paren{g_{\sigma\bar\tau}}_{\gamma\bar\delta}}
=
-g^{\bar\beta\sigma}g^{\bar\tau\alpha}
\paren{\log\det\paren{g_{\gamma\bar\delta}}}_{\sigma\bar\tau}
\\
&=
-g^{\bar\beta\sigma}g^{\bar\tau\alpha}\paren{\frac{\partial}{\partial t}g_{\sigma\bar\tau}}
-(n+1)g^{\bar\beta\sigma}g^{\bar\tau\alpha}g_{\sigma\bar\tau}
\\
&=
-g^{\bar\beta\sigma}g^{\bar\tau\alpha}\paren{\frac{\partial}{\partial t}g_{\sigma\bar\tau}}
-(n+1)g^{\bar\beta\alpha}.
\end{align*}
In the last second equality, we used the fact that the equation \eqref{E:Relative Kahler-Ricci flow} implies that
\begin{equation*}
\paren{\frac{\partial}{\partial t}g}_{\sigma\bar\tau}
=
\paren{\log\det(g_{\alpha\ov{\beta}})}_{\sigma\bar\tau}-(n+1)g_{\sigma\bar\tau}.
\end{equation*}
The equation \eqref{E:Relative Kahler-Ricci flow} also implies that 
$\Delta_t(g_{s\bar\beta})=\frac{\partial}{\partial t}g_{s\bar\beta}+(n+1)g_{s\bar\beta}$. 
Then we have
\begin{align*}
I_1
&=
g_{s\bar\beta}\Delta_t(g^{\bar\beta\alpha})g_{\alpha\bar s}
+
\Delta_t\paren{g_{s\bar\beta}}g^{\bar\beta\alpha}g_{\alpha\bar s}
\\
&=
-g_{s\bar\beta}
\paren{
	g^{\bar\beta\sigma}g^{\bar\tau\alpha}\paren{\pd{}{t}g_{\sigma\bar\tau}}
	+
	(n+1)g^{\bar\beta\alpha}
}
g_{\alpha\bar s}
+
\paren{
	\frac{\partial}{\partial t}g_{s\bar\beta}+(n+1)g_{s\bar\beta}
}
g^{\bar\beta\alpha}g_{\alpha\bar s}
\\
&=
-\paren{\pd{}{t}g_{\sigma\bar\tau}}
g^{\bar\beta\sigma}g^{\bar\tau\alpha}g_{s\bar\beta}g_{\alpha\bar s}
+
\paren{\frac{\partial}{\partial t}g_{s\bar\beta}}g^{\bar\beta\alpha}g_{\alpha\bar s},
\end{align*}
and
\begin{align*}
I_4
:=
g_{s\bar\beta}g^{\bar\beta\alpha}\Delta_t g_{\alpha\bar s}
=
g_{s\bar\beta}g^{\bar\beta\alpha}\paren{\pd{}{t}g_{\alpha\bar s}}
+
(n+1)g_{s\bar\beta}g^{\bar\beta\alpha}g_{\alpha\bar s}.
\end{align*}
Since our coordinate is normal,
$$
I_2
:=
g^{\bar\delta\gamma}
(g_{s\bar\beta}g^{\bar\beta\alpha})_\gamma(g_{\alpha\bar s})_{\bar\delta}
=
g^{\bar\delta\gamma}
(g_{s\ov{\alpha}})_\gamma 
g^{\bar\beta\alpha}(g_{\gamma\ov{s}})_{\bar\delta}
=
g^{\bar\delta\gamma}
g^{\bar\beta\alpha}(g_{s\ov{\alpha}})_\gamma
(g_{\gamma\ov{s}})_{\bar\delta}.
$$
Therefore, we have
\begin{align*}
\Delta_t c(\omega(t))
=&
I_0-I_1-I_2-I_3-I_4\\
=&
\paren{\frac{\partial}{\partial t}g}_{s\bar s}
+
(n+1)g_{s\bar s}
+
g^{\bar\delta\alpha}(g_{\alpha\bar\beta})_s
g^{\bar\beta\gamma}(g_{\gamma\bar\delta})_{\bar s}
\\
&
+\paren{\pd{}{t}g_{\sigma\bar\tau}}
g^{\bar\beta\sigma}g^{\bar\tau\alpha}g_{s\bar\beta}g_{\alpha\bar s}
-
\paren{\frac{\partial}{\partial t}g_{s\bar\beta}}g^{\bar\beta\alpha}g_{\alpha\bar s}
-
g^{\bar\delta\gamma}
g^{\bar\beta\alpha}(g_{s\ov{\alpha}})_\gamma
(g_{\gamma\ov{s}})_{\bar\delta}.
\\
&
-I_3
-
g_{s\bar\beta}g^{\bar\beta\alpha}\paren{\pd{}{t}g_{\alpha\bar s}}
-
(n+1)g_{s\bar\beta}g^{\bar\beta\alpha}g_{\alpha\bar s}
\\
=&
\paren{\paren{\pd{}{t}g}_{s\bar s}
-
g_{s\bar\beta}\paren{\pd{}{t}g^{\bar\beta\alpha}}g_{\alpha\bar s}
-
\paren{\pd{}{t}g_{s\bar\beta}}g^{\bar\beta\alpha}g_{\alpha\bar s}
-
g_{s\bar\beta}g^{\bar\beta\alpha}
\paren{\pd{}{t}g_{\alpha\bar s}}}
\\
&
+
(n+1)\paren{g_{s\bar s}-g_{s\bar\beta}g^{\bar\beta\alpha}g_{\alpha\bar s}}
-
I_3
\\
=&
\frac{\partial}{\partial t}c(\omega(t))+(n+1)c(\omega(t))-I_3.
\end{align*}
Hence it is enough to show that
$I_3=\norm{\ov{\partial}v_{\omega(t)}}^2$.
Remark \ref{R:geo_curv} says that
$$
\ov{\partial}v_{\omega(t)}
=
\paren{-(g_{s\bar\beta})_{\bar\delta}g^{\bar\beta\alpha}
-
g_{s\bar\beta}(g^{\bar\beta\alpha})_{\bar\delta}}
dz^{\bar\delta}\otimes\pd{}{z^\alpha}.
$$
In the normal coordinates, we have
\begin{align*}
\norm{\ov{\partial}v_{\omega(t)}}^2
=
g^{\bar\tau\alpha}
(g_{s\bar\tau})_{\bar\delta}
\ov{g^{\bar\sigma\beta}
(g_{s\bar\sigma})_{\bar\gamma}}
g_{\alpha\ov{\beta}}
g^{\bar\delta\gamma}
=
g^{\bar\tau\alpha}
g^{\bar\delta\gamma}
(g_{s\bar\tau})_{\bar\delta}
(g_{\bar s\alpha})_\gamma.
\end{align*}
On the other hand, 
$$
I_3
=
g^{\bar\delta\gamma}
(g_{s\bar\beta}g^{\bar\beta\alpha})_{\bar\delta}(g_{\alpha\bar s})_\gamma
=
g^{\bar\delta\gamma}
(g_{s\bar\beta})_{\bar\delta}g^{\bar\beta\alpha}(g_{\alpha\bar s})_\gamma
=
g^{\bar\delta\gamma}g^{\bar\beta\alpha}
(g_{s\bar\beta})_{\bar\delta}(g_{\alpha\bar s})_\gamma.
$$
This completes the proof.
\end{proof}

\begin{remark}
For a holomorphic family of canonically polarized compact K\"{a}hler manifolds, the positivity of $\omega(t)$ can be immediately proved by applying the standard weak and strong parabolic maximum principle to the equation \eqref{E:Evolution} (see Corollary 4.9 in \cite{Berman}).
The standard weak maximum principle, however, does not hold in general if the manifold is non-compact.
In the next section, we will use a version of weak parabolic maximum principle for non-compact manifolds due to Ni.
\end{remark}

\section{Positivity of fiberwise K\"{a}hler-Ricci flows}

Fix an arbitrary point $y\in S$. 
Denote its fiber by $\Omega:=D_y$.
Throughout this section, we will omit the index $y$ for the defining function $r_y$ and the K\"{a}hler metrics $\omega_y$ and $\omega_y(t)$.
Let $g:=-\log(-r)$ be the strictly plurisubharmonic function on $\Omega$.
Then $\omega=i\ddbar g$ is a complete K\"ahler metric on $\Omega$ satisfying
\begin{equation*}
\abs{dg}_\omega^2
=
g^{\alpha\bar\beta}g_\alpha g_{\bar\beta}
=
\frac{\abs{\partial r}^2}{\abs{\partial r}^2-r}
\leq1.
\end{equation*}
By Theorem \ref{T:Chau}, there exist one parameter family of K\"{a}hler metrics $\omega(t):=i\partial\ov\partial g(t)$ on $\Omega$ satisfying the K\"{a}hler-Ricci flow:
\begin{equation} \label{E:KRF_max}
\begin{aligned}
\pd{}{t}\omega(t)&=-\ric(\omega(t))-(n+1)\omega(t),\medskip\\
\omega(0)&=\omega.
\end{aligned}
\end{equation}
We also know that $\omega(t)$ converges to the unique complete K\"ahler-Einstein metric as $t\rightarrow\infty$. 
Moreover, there exists a constant $C>0$ (independent of $t$) such that
\begin{equation}\label{E:Quasi-iso}
\frac{1}{C}\omega
\le
\omega(t)
\le
C\omega.
\end{equation}
We denote the volume forms by $dV_t:=\frac{\omega(t)^n}{n!}$ and $dV_0:=\frac{\omega^n}{n!}$. 

\subsection{Parabolic maximum principle}
The following theorem is essentially the same with Ni's parabolic maximum principle in \cite{Ni}, except it is expressed by a plurisubharmonic exhaustion function instead of the distance function.

\begin{theorem} [cf. Theorem 2.1 in \cite{Ni}]\label{T:Maximum_Principle}
Let $f$ be a smooth function on $\Omega\times[0,T)$ satisfying
$$
\paren{\pd{}{t}-\Delta_t}f\geq0\ \ \ {\it whenever}\  f\leq0. 
$$
Assume that there exists a constant $b>0$ such that 
\begin{equation}\label{growth_control}
\int_0^{T}\int_{\Omega}(-r)^b(f_-)^2\ dV_tdt<\infty,
\end{equation}
where $f_-:=-\min\{f,0\}$.
If $f\ge0$ on $\Omega$ at $t=0$, then $f\ge0$ on $\Omega\times[0,T)$.
\end{theorem}

\begin{proof}
Let $S(t)$ be the scalar curvature of $\omega(t)$, defined by
\begin{equation*}
S(z,t):=g(t)^{\alpha\bar\beta}\paren{\log\omega(t)^n}_{\alpha\bar\beta}.
\end{equation*}
Denote by $S_*(t):=\inf\limits_{z\in \Omega}S(z,t)$.
Let $\tilde f(z,t):=\exp\paren{\int_0^t\frac{1}{2}\paren{S_*(s)+n(n+1)}ds}f(z,t)$.
A direct computation gives that
\begin{equation*}
\paren{\pd{}{t}-\Delta_t-\frac{1}{2}(S_*(t)+n(n+1))}\tilde f(z,t)\ge0
\end{equation*}
whenever $\tilde f(z,t)\le0$.
For any $T'$ with $0<T'<T$, let
\begin{equation*}
{\tilde g}(z,t)
:=
-\frac{g(z)^2}{4C(2T'-t)}.
\end{equation*}
Without the loss of the generality we may assume that $T'\le\frac{1}{2b^2C}$, since we can always split $[0,T']$ into smaller intervals (such that each has the length less than $\frac{1}{2b^2C}$) and apply the induction. 
Therefore near the boundary of $\Omega$, we have that
\begin{equation*}
e^{\tilde g}\leq e^{-\frac{b^2}{4}g^2}=e^{-(\frac{b}{2}\log(\frac{1}{-r}))^2} \leq (-r)^b,
\end{equation*}
since $e^{-x^2}\leq e^{-2x}$ for any large enough $x$. Now the condition \eqref{growth_control} implies that
\begin{equation}\label{new_growth}
\int_0^{T'}\int_{\Omega}e^{\tilde g}{\tilde f}_-^2 dV_tdt<\infty.
\end{equation}
Using the inequality \eqref{E:Quasi-iso}, we have 
$\abs{\nabla g}_{\omega(t)}^2\le C$.
Hence it follows that
\begin{equation*}
\abs{\nabla{\tilde g}}^2+\pd{}{t}{\tilde g}\le0.
\end{equation*}
Let $\chi:[0,\infty)\rightarrow[0,1]$ be a cut-off function so that
$\chi(s)=0$ for $s\ge1$ and 
$\chi(s)=1$ for $s\le1$.
Set $\eta(z):=\chi\paren{\frac{g(z)}{a}}$. Using the inequality \eqref{E:Quasi-iso},
it is easy to see that there exists a constant $C_1>0$ independent of $a$ such that
\begin{equation}\label{eta}
\abs{\nabla\eta}^2\le\frac{C_1}{a^2}.
\end{equation}
Now Stoke's theorem implies the following:
\begin{align*}
\int_{\Omega}\eta^2 e^{\tilde g} {\tilde f}_-
\Delta_t \tilde f dV_t
&=
-\int_{\Omega}
\inner{\nabla\paren{\eta^2 e^{\tilde g} {\tilde f}_-},\nabla \tilde f}
dV_t
\\
&=
-\int_{\Omega}
\paren{
	2\inner{\nabla\eta,\nabla {\tilde f}_-}\eta e^{\tilde g} {\tilde f}_-
	+
	\inner{\nabla \tilde f,\nabla{\tilde g}}\eta^2e^{\tilde g} {\tilde f}_-
	+
	\abs{\nabla{\tilde f}_-}^2\eta^2e^{\tilde g}
}
dV_t
\\
&\le
\int_{\Omega}
\paren{
	2\abs{\nabla\eta}^2e^{\tilde g}{\tilde f}_-^2
	+
	\frac{1}{2}\abs{\nabla{\tilde g}}^2\eta^2e^{\tilde g}{\tilde f}_-^2
}
dV_t.
\end{align*}
On the other hand, integration by parts implies that
\begin{equation}
\begin{aligned}
\int_0^{T'}\int_{\Omega}\eta^2 e^{\tilde g} {\tilde f}_-
&\pd{\tilde f}{t} dV_tdt
=
-\frac{1}{2}
\int_0^{T'}\int_{\Omega}
\eta^2 e^{\tilde g}\pd{}{t}\paren{{\tilde f}_-^2}
dV_tdt
\\
&=
-\int_{\Omega}
\frac{1}{2}
\eta^2 e^{\tilde g} {\tilde f}_-^2
dV_t
\Big\vert_0^{T'}
+
\frac{1}{2}
\int_0^{T'}\int_{\Omega}
\pd{}{t}\paren{\eta^2 e^{\tilde g}\frac{dV_t}{dV_0}}{\tilde f}_-^2
dV_0dt.
\end{aligned}
\end{equation}
Taking the trace of the equation \eqref{E:KRF_max}, we have
$$
\pd{g(t)_{\alpha\bar \beta}}{t}g(t)^{\alpha\bar \beta}=-S(t)-n(n+1).
$$
Using Cramer's rule, we obtain
\begin{equation*}
\pd{}{t}\paren{\frac{dV_t}{dV_0}}
=
\pd{}{t}\frac{\det(g(t)_{\alpha\bar \beta})}{\det(g_{\alpha\bar \beta})}
=
\frac{\det(g(t)_{\alpha\bar \beta})}{\det(g_{\alpha\bar \beta})}
\pd{g(t)_{\alpha\bar \beta}}{t}g(t)^{\alpha\bar \beta}
=
\frac{dV_t}{dV_0}(-S(t)-n(n+1))
\end{equation*}
Altogether, it follows that
\begin{align*}
0
&\le
\int_0^{T'}\int_{\Omega}\eta^2 e^{\tilde g} {\tilde f}_-
\paren{\pd{}{t}-\Delta_t-\frac{1}{2}(S_*(t)+n(n+1))}\tilde f\ dV_tdt
\\
&\le
\int_0^{T'}\int_{\Omega}
\paren{
	2\abs{\nabla\eta}^2e^{\tilde g}{\tilde f}_-^2
	+
	\frac{1}{2}\abs{\nabla{\tilde g}}^2\eta^2e^{\tilde g}{\tilde f}_-^2
	+
	\frac{1}{2}\pd{{\tilde g}}{t}\eta^2e^{\tilde g}{\tilde f}_-^2
}
dV_tdt \\
&\;\;\;\;\;
-
\int_{\Omega}
\frac{1}{2}
\eta^2 e^{\tilde g} {\tilde f}_-^2
dV_t
\Big\vert_0^{T'}
+
\int_0^{T'}\int_{\Omega}
\frac{1}{2}\paren{\eta^2e^{\tilde g}{\tilde f}_-^2\paren{-S(t)+S_*(t)}}
dV_tdt
\\
&\le
2\int_0^{T'}\int_{\Omega}
\abs{\nabla\eta}^2e^{\tilde g}{\tilde f}_-^2
dV_tdt 
-
\paren{
	\frac{1}{2}
	\int_{\Omega}\eta^2 e^{\tilde g} {\tilde f}_-^2
	dV_t
}
(T')
\end{align*}
Letting $a\rightarrow\infty$, the inequalities \eqref{new_growth} and \eqref{eta} imply that
\begin{equation*}
\paren{\int_{\Omega} e^{\tilde g} {\tilde f}_-^2dV_t}(T') \le 0
\end{equation*}
This implies that ${\tilde f}_-\equiv0$, therefore we have $f\ge0$ on $\Omega\times[0,T']$.
\end{proof}

\subsection{Proof of Theorem \ref{T:Main}}

By Remark \ref{R:geo_curv}, it is enough to show that the restriction of the geodesic curvature of the fiberwise K\"{a}hler-Ricci flow $c(\omega(t)):=c(\omega(t))|_{\Omega}$ is positive on $\Omega$.
We will apply Theorem \ref{T:Maximum_Principle} to the function $c(\omega(t))$ on $\Omega\times[0,\infty)$.

Note that Berman's parabolic equation \eqref{E:Evolution} says that
$$
\paren{\pd{}{t}-\Delta_t}c(\omega(t))\geq0\ \ \ {\rm whenever}\  c(\omega(t))\leq0. 
$$
On the other hand, the computation in the proof of Proposition \ref{P:initial_coho} implies that
$$
dV_0=\det(r_{\gamma\bar\delta})(-r+\abs{\partial r}^2)\Big(\frac{1}{-r}\Big)^{n+1}dV,
$$
where $dV$ is the Euclidean volume form of $\mathbb{C}^n$.
Since
$\det(r_{\gamma\bar\delta})(-r+\abs{\partial r}^2)=e^{-F_y}$ is bounded function on $\ov\Omega$, this together with the quasi-isometry \eqref{E:Quasi-iso} implies that
$$
\int_{\Omega}(-r)^bc(\omega(t))^2dV_t \lesssim \int_{\Omega}\Big(\frac{1}{-r}\Big)^{n+1-b}c(\omega(t))^2dV.
$$
To satisfy the condition \eqref{growth_control} in Theorem \ref{T:Maximum_Principle}, we only need to show that the geodesic curvature $c(\omega(t))$ has a polynomial growth near the boundary with respect to the defining function.
More precisely, we will show that $\abs{c(\omega(t))}=O((-r)^{-2})$.

First consider the initial data $c(\omega)$.
Since $\omega=i\partial\ov\partial g$ with $g:=-\log(-r)$, we have
\begin{equation}\label{E:initial_growth}
c(\omega):=\inner{v_\omega,v_\omega}_\omega
=
\frac{1}{-r}i\partial\ov\partial r(v_\omega,\overline{v_\omega}) 
+\frac{1}{r^2}\abs{\partial r(v_\omega)}^2=O((-r)^{-2}).
\end{equation}
Hence it suffices to show the following proposition.

\begin{proposition}\label{P:Estimate_geo_t}
There exists constant $C>0$ independent of $t$ such that
\begin{equation}\label{E:time_growth}
\abs{c(\omega(t))-c(\omega)}\leq\frac{C}{-r}.
\end{equation}
\end{proposition}

\begin{proof}
Recall that the fiberwise K\"{a}hler-Ricci flow $\omega(t)$ on $D$ is given by $\omega(t):=i\partial\ov\partial g(t)$ where $g(t):=g+\varphi(t)$.
For a fixed $t\in(0,\infty)$, denote by $\varphi:=\varphi(t)$. 
Under the Euclidean coordinate system $(z^1,\ldots,z^n,s)$, $c(\omega(t))$ can be expressed as
$$
c(\omega(t))
=
g_{s\bar{s}}+\varphi_{s\bar{s}}-(g_{s\bar\beta}+\varphi_{s\bar\beta}){g(t)}^{\bar\beta\alpha}(g_{\alpha\bar{s}}+\varphi_{\alpha\bar{s}}) \\
$$
Since $c(\omega)=g_{s\bar{s}}-g_{s\bar\beta}g^{\bar\beta\alpha}g_{\alpha\bar{s}}$ and $\frac{1}{C}\omega\leq\omega(t)\leq C\omega$, we have
$$
\abs{c(\omega(t))-c(\omega)}\lesssim |\varphi_{s\bar{s}}-\varphi_{s\bar\beta}g^{\bar\beta\alpha}\varphi_{\alpha\bar{s}}-g_{s\bar\beta}g^{\bar\beta\alpha}\varphi_{\alpha\bar{s}}-\varphi_{s\bar\beta}g^{\bar\beta\alpha}g_{\alpha\bar{s}}|.
$$
Moreover, an explicit calculation of derivatives of $g$ implies that $g^{\bar\beta\alpha}=O(-r)$, $g_{s\bar\beta}g^{\bar\beta\alpha}$ and $g^{\bar\beta\alpha}g_{\alpha\bar{s}}$ are bounded functions on $\Omega$ (cf. Section 5 in \cite{Choi1}). Remark \ref{R:Boundedness of u} implies that $\varphi_{s\bar{s}}$ is bounded. 
Thus it is enough to estimate functions $\varphi_{\alpha\bar{s}}$ and $\varphi_{s\bar\beta}$.
Note that Remark \ref{R:Boundedness of u} implies that
$$
\norm{\xi^{\ast}_p\varphi_s}_{C^{k+\epsilon,\frac{k+\epsilon}{2}}(V_p\times[0,\infty))}\leq C_k
$$
for some constant $C_k>0$.
In particular, this implies that there exist a constant $C>0$ independent of $t$ such that
$
\abs{\frac{\partial}{\partial \xi^{j}}\varphi_{s}}
\leq C,
$
where $\xi_p=(\xi^1,\ldots,\xi^n)$ is the coordinate system satisfying the conditions of bounded geometry.
By the construction of the coordinate system for the strongly pseudoconvex domain (see Section 1 in \cite{Cheng_Yau}), we obtain the estimate
$$
\abs{\varphi_{s\bar\beta}}
=
\abs{\frac{\partial}{\partial z^{\bar\beta}}\varphi_{s}}
\leq
\frac{C}{(-r)}\sum_{j=1}^n\abs{\pd{}{\xi^j}\varphi_s}
\leq
\frac{C}{-r}
$$
on the Euclidean coordinates $(z^1,\ldots,z^n)$.
The same argument for the function $\varphi_{\bar s}$ shows that $|\varphi_{\alpha\bar s}|\leq\frac{C}{-r}$.
This completes the proof.
\end{proof}

Equations \eqref{E:initial_growth} and \eqref{E:time_growth} imply that
$\abs{c(\omega(t))}=O((-r)^{-2})$ as we required.
Now the following strong maximum principle completes the proof of Theorem \ref{T:Main}.

\begin{theorem} [cf. Theorem  6.54 in \cite{CLN}] \label{T:Strong_Max}
Let $f$ be a smooth function on $\Omega\times[0,T)$ satisfying
$$
\paren{\pd{}{t}-\Delta_t}f\geq0. 
$$
Suppose that $f\ge0$ on $\Omega\times[0,T)$. 
If $f(x,0)>0$ for some point $x\in\Omega$ at the initial time $t=0$, then $f>0$ on $\Omega\times(0,T)$.
\end{theorem}

Finally, Theorem \ref{T:Main} and Theorem \ref{T:convergence of fiberwise} imply Corollary \ref{C:Main}.


\end{document}